\documentclass[12pt]{article}
\usepackage[centertags]{amsmath}
\usepackage{amsfonts}
\usepackage{amssymb}
\usepackage{amsthm}
\input{amssym.def}

\numberwithin{equation}{section}

\newtheorem{Definition}{Definition}[section]

\newtheorem{theorem}[Definition]{Theorem}
\newtheorem{lemma}[Definition]{Lemma}
\newtheorem{proposition}[Definition]{Proposition}
\newtheorem{corollary}[Definition]{Corollary}

\begin{document}
\title{\Large \bf On the prime spectrum of an le-module}
\author{M. Kumbhakar and A. K. Bhuniya}
\date{}
\maketitle

\begin{center}
Department of Mathematics, Nistarini College, Purulia-723101, W.B.\\
Department of Mathematics, Visva-Bharati, Santiniketan-731235, India. \\
manaskumbhakar@gmail.com, anjankbhuniya@gmail.com
\end{center}

\begin{abstract}{\footnotesize}
Here we continue to characterize a recently introduced notion, le-modules $_{R}M$ over a commutative ring $R$ with unity \cite{Bhuniya}. This article introduces and characterizes Zariski topology on the set $Spec(M)$ of all prime submodule elements of $M$. Thus we extend many results on Zariski topology for modules over a ring to le-modules. The topological space Spec(M) is connected if and only if $R/Ann(M)$ contains no idempotents other than $\overline{0}$ and $\overline{1}$. Open sets in the Zariski topology for the quotient ring $R/Ann(M)$ induces a base of quasi-compact open sets for the Zariski-topology on Spec(M). Every irreducible closed subset of Spec(M) has a generic point. Besides, we prove a number of different equivalent characterizations for Spec(M) to be spectral.
\end{abstract}



\section{Introduction}
W. Krull \cite{Krull} recognized that many properties on ideals in a commutative ring are independent on the fact that they are composed of elements. Hence those properties can be restated considering ideals to be an undivided entity or element of a suitable algebraic system. In the abstract ideal theory, the ideals are considered to be elements of a multiplicative lattice, a lattice with a commutative multiplication and satisfies some axioms. Ward and Dilworth \cite{Dilworth 1}, \cite{Dilworth 2}, \cite{Ward 1}, \cite{Ward 2}, \cite{Ward and Dilworth 1}, \cite{Ward and Dilworth 2}, \cite{Ward and Dilworth 3}, contributed many significant results in abstract ideal theory via the residuation operation on a multiplicative lattice. In \cite{Dilworth 2} Dilworth redefined principal elements and obtained the Noether primary decomposition theorems. The multiplicative theory of ideals were continued towards dimension theory, representations problem and many other overlaping areas by D. D. Anderson, E. W. Johnson, J. A. Johnson, J. P. Ladiaev, K. P. Bogart, C. Jayaram and many others. For the proper references of these articles we refer to \cite{Anderson and Johnson}. Also we refer to \cite{McCarthy} for more discussions on abstract ideal theory.

Success achieved in abstract ideal theory naturally motivated researchers to consider abstract submodule theory, which today is known as the theory of lattice modules. E. W. Johnson and J. A. Johnson \cite{Johnson1}, \cite{Johnson and Johnson 1}, introduced and studied Noetherian lattice modules. They also considered lattice modules over semilocal Noetherian lattice. Whitman \cite{Whitman 1} introduced principal elements in a lattice module and extended Nagata's principle of idealization to lattice modules. Nakkar and Anderson \cite{Nakkar and Anderson 2} studied localization in lattice modules. There are many articles devoted to lattice modules, to mention a few \cite{Nakkar and Anderson 1}, \cite{Johnson and Johnson 2}, \cite{Johnson2}, \cite{Johnson3}, \cite{Whitman 1}.

This article is a continuation of our present project on le-modules, an algebraic structure motivated by lattice modules over a multiplicative lattice. Our goal is to develop an ``abstract submodule theory'' which will be capable to give insight about rings more directly. The system we choose is a complete lattice $M$ with a commutative and associative addition which is completely join distributive and admits a module like left action of a commutative ring $R$ with $1$. Since we are taking left action of a ring $R$ not of the complete modular lattice of all ideals of $R$, we hope that influence of arithmetic of $R$ on $M$ will be easier to understand. For further details and motivation for introducing le-modules we refer to \cite{Bhuniya}.

In this article we introduce and study Zariski topology on the set Spec$(M)$ of all prime submodule elements of an le-module $_{R}M$. It is well established that the Zariski topology on prime spectrum is a very efficient tool to give geometric interpretation of the arithmetic in rings \cite {Anderson 2}, \cite{Hochster}, \cite{Lu and Yu}, \cite{Schwartz}, \cite{Zhang} and modules \cite{Abuhlail}, \cite{Behboodi}, \cite{Behboodi-Haddadi}, \cite{Duraivel}, \cite{Hassanzadeh3}, \cite{CPLu1}, \cite{CPLu2}, \cite{CPLu3}, \cite{McCasland}. Here we have extended several results on Zariski topology in modules to le-modules.

In addition to this introduction, this article comprises six sections. In Section 2, we recall definition of le-modules and various associated concepts from \cite{Bhuniya}. Also we discuss briefly on the Zariski topology in rings and modules. Section \ref{prime spectrum}, introduces Zariski topology $\tau^{*}(M)$ on the set Spec$(M)$ of all prime submodule elements of an le-module $_RM$, and proves the equivalence of a different characterization of $\tau^{*}(M)$ by the ideals of the ring $R$. Annihilator Ann$(M)$ of an le-module $_RM$ is defined so that it is an ideal of $R$. This induces a natural connection between Spec$(M)$ and Spec$(R/Ann(M))$. Section \ref{continuous-homeomorphic} establishes some basic properties of this natural connection. As a consequence, we deduce that $Spec(M)$ is connected if and only if $\overline{0}$ and $\overline{1}$ are the only idempotents of the ring $R/Ann(M)$. Section \ref{Base} finds a base for the Zariski topology on Spec$(M)$ consisting of quasi-compact open sets. Section \ref{irreducible-generic} characterizes irreducible closed subsets and generic points in Spec$(M)$ showing that these two notions are closely related. Section \ref{section spectral} proves the equivalence of a number of different characterizations for Spec(M) to be spectral.

\section{Preliminaries}    \label{preliminaries}
Throughout the article, R stands for a commutative ring with $1$. The cardinality of a set $X$ will be denoted by $|X|$.

First we recall the definition of an le-module and various associated concepts from \cite{Bhuniya}. Here by an \emph{le-semigroup} we mean $(M, +, \leqslant, e)$ such that $(M, \leq)$ is a complete lattice, $(M, +)$ is a commutative monoid with the zero element $0_M$ and for all $m, m_i \in M, i \in I$ it satisfies $m \leq e$ and
\begin{enumerate}
\item[(S)] $m+(\vee_{i \in I}m_{i}) = \vee_{i \in I}(m+m_{i})$.
\end{enumerate}

Let $R$ be a ring and $(M,+,\leqslant, e)$ be an le-semigroup. If there is a mapping $R \times M \longrightarrow M$ which satisfies
\begin{enumerate}
\item[(M1)]  $r(m_{1}+m_{2})= rm_{1}+rm_{2}$;
\item[(M2)] $(r_{1}+r_{2})m \leqslant r_{1}m+r_{2}m$;
\item[(M3)] $(r_{1}r_{2})m=r_{1}(r_{2}m)$;
\item[(M4)] $1_{R}m=m; \;\;\; 0_{R}m=r0_{M}=0_{M}$;
\item[(M5)] $r(\vee_{i \in I}m_{i}) = \vee_{i \in I}rm_{i}$,
\end{enumerate}
for all $r, r_{1}, r_{2} \in R$ and $m, m_{1}, m_{2}, m_{i} \in M, i \in I$ then $M$ is called an \emph{le-module} over $R$. It is denoted by $_RM$ or by $M$ if it is not necessary to mention the ring $R$.

From (M5), we have,
\begin{enumerate}
\item[(M5)$^{\prime}$] $m_{1} \leqslant m_{2} \Rightarrow rm_{1} \leqslant rm_{2}$, for all $r \in R$ and $m_{1},m_{2} \in$ $M$.
\end{enumerate}

An element $n$ of an le-module $_{R}M$ is said to be a \emph{submodule element} if $ n+n, rn \leqslant n$, for all $r \in R$. We call a submodule element $n$ proper if $n \neq e$. Note that $ 0_{M}= 0_{R}.n \leqslant n$, for every submodule element $n$ of $M$. Also $ n+n = n $, i.e. every submodule element of $M$  is an idempotent. We define the sum of the family $\{n_{i}\}_{i \in I}$ of submodule elements in $_RM$ by:
\begin{center}
$\sum_{i \in I}n_{i} = \vee \{ (n_{i_{1}}+n_{i_{2}}+ \cdots +n_{i_{k}}) : k \in \mathbb{N}, and \;\; i_{1}, i_{2}, \cdots, i_{k} \in I \}$.
\end{center}
Since $_{R}M$ is assumed to be complete, $\sum_{i \in I}n_{i}$ is well defined. It is easy to check that $\sum_{i \in I}n_{i}$ is a submodule element of $M$.

If $n$ is a submodule element in $_RM$, then we denote
\begin{center}
$(n:e)= \{r \in R: re \leqslant n \}$.
\end{center}
Then $0_{R} \in (n:e)$ implies that $(n:e) \neq \emptyset$. One can check that $(n:e)$ is an ideal of $R$. For submodule elements $n \leqslant l$ of an le-module $_{R}M$, we have $(n : e) \subseteq (l : e)$. Also if $\{n_{i}\}_{i \in I}$ be an arbitrary family of submodule elements in $_{R}M$, then $ (\wedge_{i \in I} n_{i} : e) = \cap_{i \in I} (n_{i} : e)$.

We call $(0_{M}:e)$ the annihilator of $_RM$. It is denoted by $Ann(M)$. Thus
\begin{align*}
Ann(M) &= \{r \in R : re \leqslant 0_{M}\}\\ &= \{r \in R : re = 0_{M}\}.
\end{align*}

For an ideal $I$ of $R$, we define
\begin{center}
$ Ie = \vee \{ \sum_{i = 1}^{k} a_{i}e : k \in \mathbb{N}; a_{1},a_{2},\cdots,a_{k} \in I \}$
\end{center}
Then $Ie$ is a submodule element of $M$. Also for any two ideals $I$ and $J$ of $R$, $I \subseteq J$ implies that $Ie \leqslant Je$. The following result, proved in \cite{Bhuniya}, is useful here.
\begin{lemma}
Let $_{R}M$ be an le-module and $n$ be a submodule element of $M$. Then for any ideal $I$ of $R$, $Ie \leqslant n$ if and only if $I \subseteq (n:e)$.
\end{lemma}

Now we recall some notions from rings and modules. An ideal $P$ in a ring $R$ is called \emph{prime} if for every $a, b \in R$, $ab \in P$ implies that $a \in P$ or $b \in P$. We denote the set of all prime ideals of a ring $R$ by $X^{R}$ or Spec$(R)$. For every ideal $I$ of $R$, we define
\begin{align*}
&V^{R}(I) = \{P \in Spec(R): I \subseteq P\}, \; \; and \\
&\tau(R) = \{X^{R}\setminus V^{R}(I) : I\; is\; an\; ideal\; of\; R\}.
\end{align*}
Then $\tau(R)$ is a topology on Spec$(R)$, which is known as the \emph{Zariski topology} on Spec$(R)$. There are many enlightening characterizations associating arithmetical properties of $R$ and topological properties of Spec$(R)$ \cite{Schwartz}.

Let $M$ be a left $R$-module. Then a proper submodule $P$ of $M$ is called a \emph{prime submodule} if for every for $r \in R$ and $n \in M$, $rn \in P$  implies that either $n \in P$ or $rM \subseteq P$. We denote the set of all prime submodules of $M$ by Spec$(M)$.

For a submodule $N$ of $M$, $(N:M) = \{r \in R: rM \subseteq N\}$ is an ideal of $R$. There is a topology $\tau(M)$ on Spec$(M)$ such that the closed subsets are of the form
\begin{center}
$V(N) = \{P \in Spec(M): (N:M) \subseteq (P:M)\}$.
\end{center}
The topology $\tau(M)$ is called the Zariski topology on $M$. Associating arithmetic of a module over a ring $R$ with the geometry of the Zariski topology on $M$ is an active area of research on modules \cite{Abuhlail}, \cite{Behboodi-Haddadi}, \cite{Duraivel}, \cite{Hassanzadeh3}, \cite{CPLu1}, \cite{CPLu2}, \cite{CPLu3}, \cite{McCasland}.

The notion of prime submodule elements was introduced in \cite{Bhuniya}, which extends prime submodules of a module over a ring. A proper submodule element $p$ of an le-module $_{R}M$ is said to be a \emph{prime submodule element} if for every $r \in R$ and $n \in M$, $rn \leqslant p$ implies that $r \in (p:e) $ or $n \leqslant p$. The \emph{prime spectrum} of $M$ is the set of all prime submodule elements of $M$ and it is denoted by Spec$(M)$ or $X^{M}$. For $P \in Spec(R)$, we denote
\begin{center}
$Spec_{P}(M) = \{p \in Spec(M): (p:e) = P\}$.
\end{center}

We also have the following relation between prime submodule elements of an le-module $_RM$ and prime ideals of $R$.
\begin{lemma}\cite{Bhuniya}     \label{prime ideal}
If $p$ is a prime submodule element of $_{R}M$, then $(p:e)$ is a prime ideal of $R$.
\end{lemma}

Also we refer to \cite{Atiyah}, \cite{Bourbaki} for background on commutative ring theory, to \cite{Bourbaki top} for fundamentals on topology and to \cite{McCarthy} for details on multiplicative theory of ideals.

\section{The Zariski topology on an le-module}     \label{prime spectrum}
In this section we give the definition and an alternative characterization of Zariski topology on the prime spectrum Spec(M) of an le-module $M$. For any submodule element $n$ of $M$, we consider two different types of varieties $V(n)$ and $V^{*}(n)$ defined by
\begin{align*}
&V(n)=\{ p \in Spec(M): n \leqslant p \}; \; \; and \\
&V^{*}(n)=\{ p \in Spec(M): (n:e) \subseteq (p:e) \}.
\end{align*}
Then $V(n) \subseteq V^{*}(n)$ for every submodule element $n$ of $M$.

\begin{proposition}      \label{Zariski topology}
Let $_{R}M$ be an le-module. Then
\begin{enumerate}
\item[(i)]
$V^{*}(0_{M})= X^{M} = V(0_{M})$;
\item[(ii)]
$V^{*}(e)=\emptyset = V(e)$;
\item[(iii)]
For an arbitrary family of submodule elements $\{n_{i}\}_{i \in I}$ of $M$, \\
(a) $\cap_{i \in I} V^{*}(n_{i})= V^{*}(\sum_{i \in I}(n_{i}:e)e)$;\\
(b) $\cap_{i \in I} V(n_{i})= V(\sum_{i \in I}n_{i})$;
\item[(iv)]
For any two submodule elements $n$ and $l$ of $M$,\\
(a) $V^{*}(n)\cup V^{*}(l) = V^{*}(n \wedge l)$;\\
(b) $V(n)\cup V(l) \subseteq V(n \wedge l)$.
\end{enumerate}
\end{proposition}
\begin{proof}
$(i)$ and $(ii)$ are obvious. Also the proofs of $(iii)(b)$ and $(iv)(b)$ are similar to $(iii)(a)$ and $(iv)(a)$ respectively. Hence we prove here only $(iii)(a)$ and $(iv)(a)$.\\
$(iii)(a)$ Let $p \in \cap_{i \in I} V^{*}(n_{i})$. Then $(n_{i}:e) \subseteq (p:e)$ implies that $(n_{i}:e)e \leqslant (p:e)e \leqslant p$, for all $i \in I$. Consequently $\sum_{i \in I}(n_{i}:e)e \leqslant p$ and so $(\sum_{i \in I}(n_{i}:e)e :e) \subseteq (p:e)$. Hence $p \in V^{*}(\sum_{i \in I}(n_{i}:e)e)$ and it follows that $\cap_{i \in I} V^{*}(n_{i}) \subseteq V^{*}(\sum_{i \in I}(n_{i}:e)e)$. Next let $p \in V^{*}(\sum_{i \in I}(n_{i}:e)e)$. Then for any $j \in I$, $(n_{j}:e) \subseteq ((n_{j}:e)e :e) \subseteq (\sum_{i \in I}(n_{i}:e)e :e) \subseteq (p:e)$ implies that $p \in V^{*}(n_{j})$ and so $V^{*}(\sum_{i \in I}(n_{i}:e)e) \subseteq \cap_{i \in I} V^{*}(n_{i})$. Thus $\cap_{i \in I} V^{*}(n_{i})= V^{*}(\sum_{i \in I}(n_{i}:e)e)$.\\
$(iv)(a)$ Now $n \wedge l \leqslant n$ and $n \wedge l \leqslant l$ implies that $V^{*}(n) \cup V^{*}(l) \subseteq V^{*}(n \wedge l)$. Let $p \in V^{*}(n \wedge l)$. Then $(n \wedge l :e) \subseteq (p:e)$ implies that $(n:e) \cap (l:e) \subseteq (p:e)$. Since $(p:e)$ is a prime ideal, either $(n:e) \subseteq (p:e)$ or $(l:e) \subseteq (p:e)$. Hence $p \in V^{*}(n) \cup V^{*}(l)$ and it follows that $V^{*}(n \wedge l) \subseteq V^{*}(n) \cup V^{*}(l)$. Therefore $V^{*}(n) \cup V^{*}(l) = V^{*}(n \wedge l)$.
\end{proof}

Thus we see that the collection $\{V(n) \mid n \; \textrm{is a submodule element of} \; M \}$ is not closed under finite unions and hence fails to be the set of all closed subsets of some topology on $X^M$. For any ideal $I$, $Ie$ is a submodule element of $M$. Now we see that the subcollection $\{V(Ie) \mid I \; \textrm{is an ideal of} \; R \}$ is closed under finite unions.
\begin{lemma}     \label{union intersection}
Let $_{R}M$ be an le-module. Then for any ideals $I$ and $J$ in $R$,\\
(i) $V(Ie) \cup V(Je) = V((I \cap J)e) = V((IJ)e)$;\\
(ii) $V^{*}(Ie) \cup V^{*}(Je) = V^{*}((I \cap J)e) = V^{*}((IJ)e)$. In particular, $V^{*}(re) \cup V^{*}(se) = V^{*}((rs)e)$ for any $r,s \in R$.
\end{lemma}
\begin{proof}
(i) First $I \cap J \subseteq I$ implies that $(I \cap J)e \leqslant Ie$ and so $V(Ie) \subseteq V((I \cap J)e)$. Similarly $V(Je) \subseteq V((I \cap J)e)$, and we have $V(Ie) \cup V(Je) \subseteq V((I \cap J)e)$. Also $IJ \subseteq I \cap J$ implies that $V((I \cap J)e) \subseteq V((IJ)e)$. Now let $p \in V((IJ)e)$. Then $(IJ)e \leqslant p$ implies that $IJ \subseteq (p:e)$. Since $(p:e)$ is a prime ideal, either $I \subseteq (p:e)$ or $J \subseteq (p:e)$. Then either $Ie \leqslant (p:e)e \leqslant p$ or $Je \leqslant (p:e)e \leqslant p$. Hence $p \in V(Ie) \cup V(Je)$ and it follows that $V((IJ)e) \subseteq V(Ie) \cup V(Je)$. Thus $V(Ie) \cup V(Je) \subseteq V((I \cap J)e) \subseteq V((IJ)e) \subseteq V(Ie) \cup V(Je)$. This completes the proof. \\
(ii) Similar.
\end{proof}

We denote,
\begin{align*}
&\mathcal{V}(M) = \{V(n): n \; is\; a\; submodule\; element\; of\; M \}, \\ &\mathcal{V^{*}}(M) = \{V^{*}(n): n\; is\; a\; submodule\; element\; of\; M \}, \\ &\mathcal{V'}(M) = \{V(Ie): I\; is\; an\; ideal\; of\; R \}.
\end{align*}
From Proposition \ref{Zariski topology}, it follows that there exists a topology, $\tau(M)$ say, on Spec$(M)$ having $\mathcal{V}(M)$ as the collection of all closed sets if and only if $\mathcal{V}(M)$ is closed under finite unions. In this case, we call the topology $\tau(M)$ the \emph{quasi-Zariski topology} on Spec$(M)$. Also from Proposition \ref{Zariski topology}, it is evident that for any le-module $_{R}M$ there always exists a topology, $\tau^{*}(M)$ say, on Spec$(M)$ having $\mathcal{V^{*}}(M)$ as the family of all closed sets. This topology $\tau^{*}(M)$ is called the \emph{Zariski topology} on Spec$(M)$. In this article we focus on the basic properties of the Zariski topology $\tau^{*}(M)$. By Lemma \ref{union intersection}, it follows that $\mathcal{V'}(M)$ induces a topology, $\tau'(M)$ say, on Spec$(M)$ for every le-module $_{R}M$.

Now we study the interrelations among these three topologies $\tau(M)$, $\tau^{*}(M)$ and $\tau'(M)$.
\begin{proposition}
Let $_{R}M$ be an le-module and $n$, $l$ are submodule elements of $M$. If $(n:e) = (l:e)$ then $V^{*}(n) = V^{*}(l)$. The converse is also true if both $n$ and $l$ are prime.
\end{proposition}
\begin{proof}
Let $p \in V^{*}(n)$. Then $(n:e) \subseteq (p:e)$, i.e, $(l:e) \subseteq (p:e)$ and hence $p \in V^{*}(l)$. Thus $V^{*}(n) \subseteq V^{*}(l)$. Similarly $V^{*}(l) \subseteq V^{*}(n)$. Therefore $V^{*}(n) = V^{*}(l)$. Conversely suppose that $V^{*}(n) = V^{*}(l)$ and both $n$ and $l$ are prime. Let $r \in (n:e)$. Then $re \leqslant n$ implies that $V^{*}(n) \subseteq V^{*}(re)$, i.e, $V^{*}(l) \subseteq V^{*}(re)$. Since $l$ is a prime submodule element, $l \in V^{*}(l)$, and so $l \in V^{*}(re)$. Thus $r \in (re:e) \subseteq (l:e)$. Therefore $(n:e) \subseteq (l:e)$. Similarly $(l:e) \subseteq (n:e)$, and hence $(n:e) = (l:e)$.
\end{proof}
\begin{proposition}       \label{V^{*}(n) = V^{*}((n:e)e)}
Let $_{R}M$ be an le-module, $n$ be a submodule element of $M$ and $I$ be an ideal of $R$. Then \\
(i) $V^{*}(n) = \cup_{P \in V^{R}((n:e))} Spec_{P}(M)$;\\
(ii) $V^{*}(n) = V^{*}((n:e)e) = V((n:e)e)$;\\
(iii) $V(Ie) = V^{*}(Ie)$. In particular $V(re) = V^{*}(re)$ for every $r \in R$.
\end{proposition}
\begin{proof}
(i) Let $p \in V^{*}(n)$. Then $(n:e) \subseteq (p:e)$ and so $p \in \cup_{P \in V^{R}((n:e))} Spec_{P}(M)$, since $(p:e)$ itself a prime ideal. Thus $V^{*}(n) \subseteq \cup_{P \in V^{R}((n:e))} Spec_{P}(M)$. Also let $p \in \cup_{P \in V^{R}((n:e))} Spec_{P}(M)$. Then there exists a prime ideal $P_{0} \in V^{R}((n:e))$ such that $p \in Spec_{P_{0}}(M)$. This implies that $(n:e) \subseteq P_{0} = (p:e)$, i.e, $p \in V^{*}(n)$. Hence $\cup_{P \in V^{R}((n:e))} Spec_{P}(M) \subseteq V^{*}(n)$. Therefore $V^{*}(n) = \cup_{P \in V^{R}((n:e))} Spec_{P}(M)$.\\
(ii) Since $(n:e)e \leqslant n$, $V^{*}(n) \subseteq V^{*}((n:e)e)$. Let $p \in V^{*}((n:e)e)$. Then $((n:e)e:e) \subseteq (p:e)$. Now $(n:e) \subseteq ((n:e)e:e)$ implies that $(n:e) \subseteq (p:e)$ and so $p \in V^{*}(n)$. Thus $V^{*}((n:e)e) \subseteq V^{*}(n)$ and hence $V^{*}(n) = V^{*}((n:e)e)$. Let $p \in V^{*}(n)$. Then $(n:e) \subseteq (p:e)$ which implies that $(n:e)e \leqslant (p:e)e \leqslant p$, i.e, $p \in V((n:e)e)$. Thus $V^{*}(n) \subseteq V((n:e)e)$. Also let $p \in V((n:e)e)$. Then $(n:e)e \leqslant p$ implies $((n:e)e:e) \subseteq (p:e)$. Thus $p \in V^{*}((n:e)e) = V^{*}(n)$ and hence $V((n:e)e) \subseteq V^{*}(n)$. Therefore $V^{*}(n) = V((n:e)e) = V^{*}((n:e)e)$.\\
(iii) The proof is omitted since it is easy to prove.
\end{proof}

\begin{theorem}      \label{identical}
For any le-module $_{R}M$, the Zariski topology $\tau^{*}(M)$ on Spec$(M)$ is identical with $\tau'(M)$.
\end{theorem}
\begin{proof}
It is suffices to prove that $\mathcal{V^{*}}(M) = \mathcal{V'}(M)$. Let $V^{*}(n)$ be a closed set in $\mathcal{V^{*}}(M)$ for some submodule element $n$ of $M$. Then by Proposition \ref{V^{*}(n) = V^{*}((n:e)e)}, $V^{*}(n) = V((n:e)e) = V(Ie)$, where $(n:e) = I$, an ideal of $R$. Thus every closed set in $\mathcal{V^{*}}(M)$ is a closed set in $\mathcal{V'}(M)$ and hence $\mathcal{V^{*}}(M) \subseteq \mathcal{V'}(M)$. Now let $V(Ie)$ be a closed set in $\mathcal{V'}(M)$ for some ideal $I$ of $R$. Again by Proposition \ref{V^{*}(n) = V^{*}((n:e)e)}, $V(Ie) = V^{*}(Ie)$. Since $Ie$ is a submodule element of $M$, $V(Ie)$ is a closed set in $\mathcal{V^{*}}(M)$. Thus $\mathcal{V'}(M) \subseteq \mathcal{V^{*}}(M)$. Therefore $\mathcal{V^{*}}(M) = \mathcal{V'}(M)$.
\end{proof}

An le-module $M$ is called a \emph{top le-module} if $\mathcal{V}(M)$ is closed under finite unions. Hence if $M$ is a top le-module, then $\tau(M)=\{X^M \setminus V(n) \mid n \; is\; a\; submodule\; element\; of\; M \}$ becomes a topology on $X^M$.
\begin{theorem}      \label{finer}
For any top le-module $_{R}M$, the quasi-Zariski topology $\tau(M)$ on Spec$(M)$ is finer than the Zariski topology $\tau^{*}(M) = \tau'(M)$.
\end{theorem}
\begin{proof}
Clearly $\mathcal{V'}(M) \subseteq \mathcal{V}(M)$. Then by Theorem \ref{identical}, $\mathcal{V^{*}}(M) = \mathcal{V'}(M) \subseteq \mathcal{V}(M)$ and hence the quasi-Zariski topology $\tau(M)$ on Spec$(M)$ is finer than the Zariski topology $\tau^{*}(M)$.
\end{proof}

\section{Relation between Spec(M) and Spec(R/Ann(M))}      \label{continuous-homeomorphic}
Let $_{R}M$ be an le-module. Then Ann$(M)$ is an ideal of $R$, which allows us to consider the quotient ring $\overline{R} = R/Ann(M)$. The image of every element $r \in R$ and every ideal $I$ of $R$ such that $Ann(M) \subseteq I$ under the canonical epimorphism $\phi: R\rightarrow R/Ann(M)$ will be denoted by $\overline{r}$ and $\overline{I}$, respectively. Then for every prime ideal $P$ of $R$ and $Ann(M) \subseteq P$, the ideal $\overline{P}$ is prime in $\overline{R}$. Hence the mapping $\psi: X^{M}\rightarrow X^{\overline{R}}$ defined by
\begin{center}
$\psi(p) = \overline{(p:e)}$ for every $p \in X^{M}$
\end{center}
is well defined. We call $\psi$ the natural map on $X^{M}$.

In this section we study relationship of $X^{M}$ and $X^{\overline{R}}$ under the natural map. Here we are interested in conditions under which $\psi$ is injective, surjective, open, closed, and homeomorphic.

\begin{proposition}       \label{continuous}
For any le-module $M$, the natural map $\psi$ of $X^{M}$ is continuous for the Zariski topologies; more precisely , $\psi^{-1}(V^{\overline{R}}(\overline{I})) = V(Ie)$ for every ideal $I$ of $R$ containing Ann$(M)$.
\end{proposition}
\begin{proof}
Let $I$ be an ideal of $R$ containing Ann$(M)$ and let $p \in \psi^{-1}(V^{\overline{R}}(\overline{I})$. Then there exists some $\overline{J} \in V^{\overline{R}}(\overline{I})$ such that $\psi(p) = \overline{J}$, i.e, $(p:e)/Ann(M) = J/Ann(M)$. This implies that $(p:e) = J \supseteq I$ and so $Ie \leqslant (p:e)e \leqslant p$. Hence $p \in V(Ie)$. Therefore $\psi^{-1}(V^{\overline{R}}(\overline{I})) \subseteq V(Ie)$. Now let $q \in V(Ie)$. Then $I \subseteq (Ie:e) \subseteq (q:e)$ implies that $\overline{I} = I/Ann(M) \subseteq (q:e)/Ann(M) = \overline{(q:e)}$. Hence $q \in \psi^{-1}(V^{\overline{R}}(\overline{I}))$. Thus $V(Ie) \subseteq \psi^{-1}(V^{\overline{R}}(\overline{I}))$. Therefore $\psi^{-1}(V^{\overline{R}}(\overline{I})) = V(Ie)$.
\end{proof}
\begin{proposition}      \label{injective}
The following conditions are equivalent for any le-module $_{R}M$:\\
(i) The natural map $\psi: X^{M}\rightarrow X^{\overline{R}}$ is injective;\\
(ii) For every $p, q \in X^{M}$, $V^{*}(p) = V^{*}(q)$ implies that $p =q$;\\
(iii) $|Spec_P(M)| \leqslant 1$ for every $p \in Spec(R)$.
\end{proposition}
\begin{proof}
(i)$\Rightarrow$ (ii): Let $V^{*}(p) = V^{*}(q)$. Then $(p:e) = (q:e)$ which implies that $(p:e)/Ann(M) = (q:e)/Ann(M)$. Thus $\psi(p) = \psi(q)$ and hence $p =q$, since $\psi$ is injective.\\
(ii)$\Rightarrow$ (iii): Let $p, q \in Spec_{P}(M)$, where $P \in Spec(R)$. Then $(p:e) = P = (q:e)$ which implies that $V^{*}(p) = V^{*}(q)$. Hence $p = q$, by (ii).\\
(iii)$\Rightarrow$ (i): Let $p, q \in X^{M}$ be such that $\psi(p) = \psi(q)$. Then $(p:e)/Ann(M) = (q:e)/Ann(M)$. This implies that $(p:e) = (q:e) = P$, say. Thus $p, q \in Spec_P(M)$ and so $p =q$, by (iii). Therefore $\psi$ is injective.
\end{proof}

\begin{theorem}       \label{open and closed}
Let $_{R}M$ be an le-module and $\psi: X^{M}\rightarrow X^{\overline{R}}$ be the natural map of $X^{M}$. If $\psi$ is surjective, then $\psi$ is both closed and open. More precisely, for every submodule element $n$ of $M$, $\psi(V^{*}(n)) = V^{\overline{R}}\overline{(n:e)}$ and $\psi(X^{M}-V^{*}(n)) = X^{\overline{R}} - V^{\overline{R}}\overline{(n:e)}$.
\end{theorem}
\begin{proof}
By the Theorem \ref{continuous}, we have $\psi$ is a continuous map and $\psi^{-1}(V^{\overline{R}}(\overline{I})) = V(Ie)$, for every ideal $I$ of $R$ containing $Ann(M)$. Thus for every submodule element $n$ of $M$, $\psi^{-1}(V^{\overline{R}}(\overline{n:e})) = V((n:e)e) = V^{*}(n)$. This implies that $\psi(V^{*}(n)) = \psi o \psi^{-1} (V^{\overline{R}}(\overline{n:e})) = V^{\overline{R}}(\overline{n:e})$, since $\psi$ is surjective. Similarly $\psi(X^{M}-V^{*}(n)) = \psi(\psi^{-1}(X^{\overline{R}}) - \psi^{-1}(V^{\overline{R}}(\overline{n:e}))) = \psi(\psi^{-1}(X^{\overline{R}} - V^{\overline{R}}(\overline{n:e}))) = \psi o \psi^{-1}(X^{\overline{R}} - V^{\overline{R}}(\overline{n:e})) = X^{\overline{R}} - V^{\overline{R}}(\overline{n:e})$. Thus $\psi$ is both closed and open.
\end{proof}
\begin{corollary}      \label{homeomorphic}
Let $_{R}M$ be an le-module and $\psi: X^{M}\rightarrow X^{\overline{R}}$ be the natural map of $X^{M}$. Then $\psi$ is bijective if and only if $\psi$ is homeomorphic.
\end{corollary}

A commutative ring R with 1 is said to be a \emph{quasi-local ring} if it has a unique maximal ideal.
\begin{theorem}
Let $_{R}M$ be an le-module and $\psi: X^{M}\rightarrow X^{\overline{R}}$ be the surjective natural map of $X^{M}$. Then the following statements are equivalent:\\
(i) $X^{M} = Spec(M)$ is connected;\\
(ii) $X^{\overline{R}} = Spec(\overline{R})$ is connected;\\
(iii) The ring $\overline{R}$ contains no idempotent other than $\overline{0}$ and $\overline{1}$.\\
Consequently, if either $R$ is a quasi-local ring or $Ann(M)$ is a prime ideal of $R$, then both $X^{M}$ and $X^{\overline{R}}$ are connected.
\end{theorem}
\begin{proof}
(i) $\Rightarrow$ (ii): From Theorem \ref{continuous}, we have that $\psi$ is a continuous map. Then (ii) follows from the fact that $\psi$ is surjective and continuous image of a connected space is connected.\\
(ii) $\Rightarrow$ (i): Let $X^{\overline{R}}$ be connected. If possible assume that $X^{M}$ is disconnected. Then $X^{M}$ must contain a non-empty proper subset $Y$ which is both open and closed. By Theorem \ref{open and closed}, $\psi(Y)$ is a non-empty subset of $X^{\overline{R}}$ that is both open and closed. We assert that $\psi(Y)$ is a proper subset of $X^{\overline{R}}$. Since $Y$ is open, $Y = X^{M} - V^{*}(n)$ for some submodule element $n$ of $M$. Then by Theorem \ref{open and closed}, $\psi(Y) = \psi(X^{M} - V^{*}(n)) = X^{\overline{R}} - V^{\overline{R}}\overline{(n:e)}$. Therefore, if $\psi(Y) = X^{\overline{R}}$, then $V^{\overline{R}}\overline{(n:e)} = \emptyset$. Now suppose that $\overline{(n:e)} \neq \overline{R}$. Then $\overline{(n:e)}$ is a proper ideal of $\overline{R}$ and so contained in a maximal ideal, say $\overline{P}$ of $\overline{R}$, which is also a prime ideal of $\overline{R}$. Thus $\overline{(n:e)} \subseteq \overline{P}$ and hence $\overline{P} \in V^{\overline{R}}\overline{(n:e)}$, i.e, $V^{\overline{R}}\overline{(n:e)} \neq \emptyset$, a contradiction. Thus $\overline{(n:e)} = \overline{R}$, i.e, $n =e$. This implies that $Y = X^{M} - V^{*}(n) = X^{M} - V^{*}(e) = X^{M}$, which is an absurd since $Y$ is a proper subset of $X^{M}$. Thus $\psi(Y)$ is a proper subset of $X^{\overline{R}}$ and hence $X^{\overline{R}}$ is disconnected, a contradiction. Therefore $X^{M} = Spec(M)$ is connected.

The equivalence of (ii) and (iii) is well-known \cite{Bourbaki}.
\end{proof}

\section{A base for the Zariski topology on Spec(M)}       \label{Base}
For any element $r$ of a ring $R$, the set $D_{r} = X^{R} - V^{R}(rR)$ is open in $X^{R}$ and the family $\{D_{r} : r \in R\}$ forms a base for the Zariski topology on $X^{R}$. Each $D_{r}$, in particular $D_{1} = X^{R}$ is known to be quasi-compact. In \cite{CPLu1}, Chin-Pi Lu, introduced a base for the Zariski topology on Spec$(M)$ for any $R$-module $M$, which is similar to that on $X^{R}$. In this section, we introduce a base for the Zariski topology on $X^{M}$ for any le-module $_{R}M$.

For each $r \in R$ we define,
\begin{center}
$X_{r} = X^{M} - V^{*}(re)$.
\end{center}
Then every $X_{r}$ is an open set in $X^{M}$. Note that $X_{0} = \emptyset$ and $X_{1} = X^{M}$.

\begin{proposition}     \label{D_{r}}
Let $_{R}M$ be an le-module with the natural map $\psi: X^{M}\rightarrow X^{\overline{R}}$ and $r \in R$. Then\\
(i) $\psi^{-1}(D_{\overline{r}}) = X_{r}$; and \\
(ii) $\psi(X_{r}) \subseteq D_{\overline{r}}$; the equality holds if $\psi$ is surjective.
\end{proposition}
\begin{proof}
(i):  $\psi^{-1}(D_{\overline{r}}) = \psi^{-1}(X^{\overline{R}} - V^{\overline{R}}(\overline{r}\overline{R})) = \psi^{-1}(X^{\overline{R}}) - \psi^{-1}(V^{\overline{R}}(\overline{r}\overline{R})) = \psi^{-1}(V^{\overline{R}}(\overline{0})) - \psi^{-1}(V^{\overline{R}}(\overline{rR})) = V(0_{M}) - V(rRe) = X^{M} - V(re) = X^{M} - V^{*}(re) = X_{r}$, by Proposition \ref{continuous} and Proposition \ref{V^{*}(n) = V^{*}((n:e)e)}(iii).\\
(ii) follows from (i).
\end{proof}

Now we have a useful lemma which will be used in the next theorem:
\begin{lemma}     \label{X_{rs}}
Let $_{R}M$ be an le-module.\\
(i) For every $r,s \in R$, $X_{rs} = X_{r} \cap X_{s}$.\\
(ii) For any ideal $I$ in $R$, $V^{*}(Ie) = \cap_{a \in I}V^{*}(ae)$.
\end{lemma}
\begin{proof}
(i) $X_{rs} = X^{M} - V^{*}((rs)e) = X^{M} - (V^{*}(re) \cup V^{*}(se)) = (X^{M} - V^{*}(re)) \cap (X^{M} - V^{*}(se)) = X_{r} \cap X_{s}$, by Lemma \ref{union intersection}(ii).\\
(ii) Let $p \in V^{*}(Ie)$. Then $(Ie:e) \subseteq (p:e)$. Now for all $a \in I$, $ae \leqslant Ie$ implies that $(ae:e) \subseteq (Ie:e) \subseteq (p:e)$. Thus $p \in V^{*}(ae)$ for all $a \in I$ and so $p \in \cap_{a \in I}V^{*}(ae) $. Hence $V^{*}(Ie) \subseteq \cap_{a \in I}V^{*}(ae)$. Also let $p \in \cap_{a \in I}V^{*}(ae)$. Then for all $a \in I$, $p \in V^{*}(ae) = V(ae)$, by Proposition \ref{V^{*}(n) = V^{*}((n:e)e)}, which implies that $ae \leqslant p$. Thus for any $k \in \mathbb{N}$ and $a_{1}, a_{2}, \cdots, a_{k} \in I$, $a_{1}e + a_{2}e + \cdots + a_{k}e \leqslant p$ and hence $Ie \leqslant p$. Then $(Ie:e) \subseteq (p:e)$ and so $p \in V^{*}(Ie)$. Hence $\cap_{a \in I}V^{*}(ae) \subseteq V^{*}(Ie)$. Therefore $V^{*}(Ie) = \cap_{a \in I}V^{*}(ae)$.
\end{proof}

\begin{theorem}      \label{Base for X^{M}}
Let $_{R}M$ be an le-module. Then the set $B = \{X_{r}: r \in R\}$ forms a base for the Zariski topology on $X^{M}$ which may be empty.
\end{theorem}
\begin{proof}
If $X^{M} = \emptyset$, then $B = \emptyset$ and theorem is trivially true in this case. Let $X^{M} \neq \emptyset$ and $U$ be an any open set in $X^{M}$. Then $U = X^{M} - V^{*}(Ie)$ for some ideal $I$ of $R$ since $\mathcal{V^{*}}(M) = \mathcal{V'}(M) = \{V^{*}(Ie) = V(Ie): I$ is an ideal of $R\}$, by Proposition \ref{V^{*}(n) = V^{*}((n:e)e)}. By above lemma $V^{*}(Ie) = \cap_{a \in I}V^{*}(ae)$. Hence $U = X^{M} - V^{*}(Ie) = X^{M} - \cap_{a \in I}V^{*}(ae) = \cup_{a \in I}(X^{M} - V^{*}(ae)) = \cup_{a \in I} X_{a}$. Thus $B$ is a base for the Zariski topology on $X^{M}$.
\end{proof}

A topological space $T$ is called \emph{quasi-compact} if every open cover of $T$ has a finite subcover.  Every finite space is quasi-compact, and more generally every space in which there is only a finite number of open sets is quasi-compact. A subset $Y$ of a topological space $T$ is said to be quasi-compact if the subspace $Y$ is quasi-compact. By a \emph{quasi-compact open subset} of $T$ we mean an open subset of $T$ which is quasi-compact. To avoid ambiguity, we would like to mention that a compact topological space is a quasi-compact Hausdorff space.
Quasi-compact spaces are of use  mainly in applications of topology to algebraic geometry and are seldom featured in other mathematical theories, where on the contrary compact spaces play an important role in different branches of mathematics. To keep uniformity in terminology we continue with the term quasi-compact.
\begin{theorem}       \label{open base}
Let $_{R}M$ be an le-module and the natural map $\psi: X^{M} \rightarrow X^{\overline{R}}$ is surjective. Then the following statements hold:\\
(i) The open set $X_{r}$ in $X^{M}$ for each $r \in R$ is quasi-compact. In particular, the space $X^{M}$ is quasi-compact.\\
(ii) The quasi-compact open sets of $X^{M}$ are closed under finite intersection and form an open base.
\end{theorem}
\begin{proof}
(i) Since $B = \{X_{r}: r \in R\}$ forms a base for the zariski topology on $X^{M}$ by Theorem \ref{Base for X^{M}}, for any open cover of $X_{r}$, there is a family $\{r_{\lambda}: \lambda \in \Lambda \}$ of elements of $R$ such that $X_{r} \subseteq \cup_{\lambda \in \Lambda} X_{r_{\lambda}}$. By Proposition \ref{D_{r}}.$(ii)$, $D_{\overline{r}} = \psi(X_{r}) \subseteq \cup_{\lambda \in \Lambda} \psi(X_{r_{\lambda}}) = \cup_{\lambda \in \Lambda} D_{\overline{r_{\lambda}}}$. Since $D_{\overline{r}}$ is quasi-compact, there exists a finite subset $\Lambda'$ of $\Lambda$ such that $D_{\overline{r}} \subseteq \cup_{\lambda \in \Lambda'} D_{\overline{r_{\lambda}}}$. Hence $X_{r} = \psi^{-1}(D_{\overline{r}}) \subseteq \cup_{\lambda \in \Lambda'} X_{r_{\lambda}}$, by Proposition \ref{D_{r}}.$(i)$. Thus for each $r \in R$, $X_{r}$ is quasi-compact.\\
(ii) To prove the theorem it suffices to prove that the intersection of two quasi-compact open sets of $X^{M}$ is a quasi-compact set. Let $C = C_{1} \cap C_{2}$, where $C_{1}$, $C_{2}$ are quasi-compact open sets of $X^{M}$. Since $B = \{X_{r} : r \in R\}$ is an open base for the Zariski topology on $X^{M}$, each $C_{i}$, $i =1,2$, is a finite union of members of $B$. Then by Proposition \ref{X_{rs}}, it follows that $C$ is also a finite union of members of $B$. Let $C = \cup_{i =1}^{n} X_{r_{i}}$ and let $\Omega$ be any open cover of $C$. Then $\Omega$ also covers each $X_{r_{i}}$ which is quasi-compact by (i). Hence each $X_{r_{i}}$, has a finite subcover of $\Omega$ and so does $C$. Thus $C$ is quasi-compact. The other part of the theorem follows from the existence of the open base $B$.
\end{proof}

\section{Irreducible closed subsets and generic points}     \label{irreducible-generic}
For each subset $Y$ of $X^{M}$, we denote the closure of $Y$ in Zariski topology on $X^{M}$ by $\overline{Y}$, and meet of all elements of $Y$ by $\Im(Y)$, i.e. $\Im(Y) = \wedge_{p \in Y} p$. One can check that $\Im(Y)$ is a submodule element of $M$. For each subset $Y$ of $Spec(R)$, we denote the intersection $\cap_{P \in Y}P$ of all elements of $Y$ by $\Im^R(Y)$

\begin{proposition}     \label{closed}
Let $_{R}M$ be an le-module and $Y \subseteq X^{M}$. Then $V^{*}(\Im(Y)) = \overline{Y}$. Hence $Y$ is closed if and only if $V^{*}(\Im(Y)) = Y$.
\end{proposition}
\begin{proof}
To prove the result it is sufficient to prove that $V^{*}(\Im(Y))$ is the smallest closed subset of $X^{M}$ containing $Y$. Now for all $p \in Y$, $\Im(Y) = \wedge_{p \in Y} p \leqslant p$ implies that $(\Im(Y):e) \subseteq (p:e)$, i.e., $p \in V^{*}(\Im(Y))$. Hence $Y \subseteq V^{*}(\Im(Y))$. Now let $V^{*}(n)$ be any closed subset of $X^{M}$ such that $Y \subseteq V^{*}(n)$. Then for every $p \in Y$, $(n:e) \subseteq (p:e)$ and so $(n:e) \subseteq \cap_{p \in Y} (p:e) = (\wedge_{p \in Y} p:e) = (\Im(Y):e)$. Also let $q \in V^{*}(\Im(Y))$. Then $(n:e) \subseteq (\Im(Y):e) \subseteq (q:e)$ implies that $q \in V^{*}(n)$. Thus $V^{*}(\Im(Y)) \subseteq V^{*}(n)$. Therefore $V^{*}(\Im(Y)) = \overline{Y}$.
\end{proof}

For an le-module $_{R}M$, we denote $\Phi = \{(p:e)| p \in X^{M} \}$. Then $\Phi \subseteq X^{R}$, by Lemma \ref{prime ideal}. We say $P$ is a maximal element of $\Phi$ if for any $Q \in \Phi$, $P \subseteq Q$ implies that $P = Q$. Recall that a topological space is a $T_{1}$-space if and only if every singleton subset is closed.

\begin{proposition}       \label{closure}
Let $_{R}M$ be an le-module and $p \in X^{M}$. Then
\begin{enumerate}
\item[(i)] $\overline{\{p\}} = V^{*}(p)$;
\item[(ii)] For any $q \in X^{M}$, $q \in \overline{\{p\}}$ if and only if $(p:e) \subseteq (q:e)$ if and only if $V^{*}(q) \subseteq V^{*}(p)$;
\item[(iii)] The set $\{p\}$ is closed in $X^{M}$ if and only if \\
$(a)$ $P = (p:e)$ is a maximal element of $\Phi$, and \\
$(b)$ $Spec_{P}(M) = \{p\}$, i.e, $|Spec_{P}(M)| = 1$;
\item[(iv)] $Spec(M)$ is a $T_{1}$-space if and only if \\
(a) $P = (p:e)$ is a maximal element of $\Phi$ for every $p \in X^{M}$, and \\
(b) $|Spec_{P}(M)| \leqslant 1$ for every $P \in Spec(R)$.
\end{enumerate}
\end{proposition}
\begin{proof}
(i) It follows from Proposition \ref{closed} by taking $Y = \{p\}$. \\
(ii) It is an obvious result of (i). \\
(iii) Let $\{p\}$ be closed in $X^{M}$. Then by (i), $\{p\} = V^{*}(p)$. To show $P = (p:e)$ is a maximal element of $\Phi$, let $Q \in \Phi$ be such that $P \subseteq Q$. Since $Q \in \Phi$ there is a prime submodule element $q$ of $X^{M}$ such that $(q:e) = Q$. Then $(p:e) \subseteq (q:e)$ which implies that $q \in V^{*}(p) = \{p\}$. Thus $p = q$ and so $P = Q$. For (ii) suppose that $q$ be any element of Spec$_{P}(M)$. Then $(q:e) = P =(p:e)$ implies that $q \in V^{*}(p) = \{p\}$ and hence $q = p$.

Conversely, we assume that the conditions (a) and (b) hold. Let $q \in V^{*}(p)$. Then $(p:e) \subseteq (q:e)$ and so $P = (p:e) = (q:e)$, by (a). Now (b) implies that $p = q$, i.e, $V^{*}(p) \subseteq \{p\}$. Also $\{p\} \subseteq V^{*}(p)$. Thus $\{p\} = V^{*}(p)$, i.e, $\{p\}$ is closed in $X^{M}$ by (i). \\
(iv) Note that (b) is equivalent to that $|Spec_{P}(M)| = 1$ for every $P \in \Phi$. Thus, by (iii), it follows that $\{p\}$  is closed in $X^{M}$ for every $p \in X^{M}$. Hence $X^{M}$ is a $T_{1}$-space.
\end{proof}

A topological space $T$ is called \emph{irreducible} if for every pair of closed subsets $T_{1}, T_{2}$ of $T$, \;\; $T = T_{1} \cup T_{2}$ implies $T = T_{1}$ or $T = T_{2}$. A subset $Y$ of $T$ is \emph{irreducible} if it is irreducible as a subspace of $T$. By an \emph{irreducible component} of a topological space $T$ we mean a maximal irreducible subset of $T$. Also if a subset $Y$ of a topological space $T$ is irreducible, then its closure $\overline{Y}$ is so. Since every singleton subset of $X^{M}$ is irreducible, its closure is also irreducible. Now by Proposition \ref{closure}, we have the following result:
\begin{corollary}      \label{irreducible closed subset}
$V^{*}(p)$ is an irreducible closed subset of $X^{M}$ for every prime submodule element $p$ of an le-module $_{R}M$.
\end{corollary}

It is well known that a subset $Y$ of Spec$(R)$ for any ring $R$ is irreducible if and only if $\Im^R(Y)$ is a prime ideal of $R$ \cite{Bourbaki}. Let $M$ be a left $R$-module. Then a subset $Y$ of Spec$(M)$ is irreducible if $\Im^M(Y) = \cap_{P \in Y}P$ is a prime submodule of $M$, but the converse is not true in general \cite{CPLu1}. In the following result we show that the situation in an le-module $_{R}M$ is similar to the modules over a ring. Interestingly, the converse of this result in an le-module $_{R}M$ is directly associated with the ring $R$.

\begin{proposition}      \label{irreducible-prime}
Let $_{R}M$ be an le-module and $Y \subseteq X^{M}$. If $\Im(Y)$ is a prime submodule element of $M$ then $Y$ is irreducible. Conversely, if $Y$ is irreducible then $\Psi = \{(p:e)| p \in Y \}$ is an irreducible subset of Spec$(R)$, i.e, $\Im^R(\Psi) = (\Im(Y):e)$ is a prime ideal of $R$.
\end{proposition}
\begin{proof}
Let $\Im(Y)$ be a prime submodule element of $M$. Then by Corollary \ref{irreducible closed subset}, $V^{*}(\Im(Y)) = \overline{Y}$ is irreducible and hence $Y$ is irreducible. Conversely, suppose that $Y$ is irreducible. Since $\psi$ is continuous by Proposition \ref{continuous}, the image $\psi(Y)$ of $Y$ under the natural map $\psi$ of $X^{M}$ is an irreducible subset of $X^{\overline{R}}$. Hence $\Im^{\overline{R}}(\psi(Y)) = \overline{(\Im(Y):e)}$ is a prime ideal of $X^{\overline{R}}$. Thus $\Im^R(\Psi) = (\Im(Y):e)$ is a prime ideal of $R$ so that $\Psi$ is an irreducible subset of Spec$(R)$.
\end{proof}

Also we have some other characterization of the irreducible subsets of $X^{M}$.
\begin{proposition}
Let $_{R}M$ be an le-module. Then the following statements hold:
\begin{enumerate}
\item[(i)] If $Y = \{p_{i}|i \in I \}$ is a family of prime submodule elements of $M$ which is totally ordered by ``$\leqslant$", then $Y$ is irreducible in $X^{M}$.
\item[(ii)] If Spec$_{P}(M) \neq \emptyset$ for some $P \in Spec(R)$, then \\
(a) Spec$_{P}(M)$ is irreducible, and \\
(b) Spec$_{P}(M)$ is an irreducible closed subset of $X^{M}$ if $P$ is a maximal ideal of $R$.
\item[(iii)] Let $Y \subseteq X^{M}$ be such that $(\Im(Y):e) = P$ is a prime ideal of $R$. Then $Y$ is irreducible if Spec$_{P}(M) \neq \emptyset$.
\end{enumerate}
\end{proposition}
\begin{proof}
(i) It is suffices to show that $\Im(Y) = \wedge_{i \in I} p_{i}$ is a prime submodule element. Let $r \in R$ and $n \in M$ be such that $n \nleq \Im(Y)$ and $r \notin (\Im(Y):e) = (\wedge_{i \in I} p_{i}:e) = \cap_{i \in I}(p_{i}:e)$. Then there exists $l$ and $k$ such that $n \nleq p_{l}$ and $r \notin (p_{k}:e)$. Since $Y$ is totally ordered, either $p_{l} \leqslant p_{k}$ or $p_{k} \leqslant p_{l}$. Let $p_{l} \leqslant p_{k}$. Then $r \notin (p_{l}:e)$ and $n \nleq p_{l}$ implies that $rn \nleq p_{l}$ since $p_{l}$ is a prime submodule element. Thus $rn \nleq \wedge_{i \in I} p_{i} = \Im(Y)$. Hence $\Im(Y)$ is a prime submodule element.\\
(ii) Let Spec$_{P}(M) \neq \emptyset$ for some $P \in Spec(R)$. Then $\Im(Spec_{P}(M))$ is a proper submodule element of $M$.\\
(a) Assume that $r \in R$ and $n \in M$ be such that $rn \leqslant \Im(Spec_{P}(M))$ and $n \nleqslant \Im(Spec_{P}(M))$. Since $n \nleqslant \Im(Spec_{P}(M))$, there exists $p \in X^{M}$ with $(p:e) = P$ such that $n \nleqslant p$. Now $rn \leqslant \Im(Spec_{P}(M)) \leqslant p$ and $n \nleqslant p$ implies that $re \leqslant p$, since $p$ is a prime submodule element. Thus $r \in (p:e) = P = (\Im(Spec_{P}(M)):e)$. Therefore $\Im(Spec_{P}(M))$ is a prime submodule element and hence Spec$_{P}(M)$ is irreducible.\\
(b) We prove that $Spec_{P}(M) = V^{*}(Pe)$ so that Spec$_{P}(M)$ is closed. Let $q \in Spec_{P}(M)$. then $(q:e) = P \subseteq (Pe:e)$. Since $P$ is a maximal ideal of $R$, $(Pe:e) = P = (q:e)$, and so $q \in V^{*}(Pe)$. Thus $Spec_{P}(M) \subseteq V^{*}(Pe)$. Also let $q \in V^{*}(Pe)$. Then $P \subseteq (Pe:e) \subseteq (q:e)$ which implies that $(q:e) = P$, since $P$ is a maximal ideal. Thus $q \in Spec_{P}(M)$ and hence $V^{*}(Pe) \subseteq Spec_{P}(M)$. Therefore $Spec_{P}(M) = V^{*}(Pe)$. \\
(iii) Let $p$ be any element of Spec$_{P}(M)$. Then $(p:e) = P = (\Im(Y):e)$ which implies that $V^{*}(p) = V^{*}(\Im(Y)) = \overline{{Y}}$, by Proposition \ref{closed}. Thus $\overline{{Y}}$ is irreducible and hence $Y$ is irreducible.
\end{proof}

Let $Y$ be a closed subset of a topological space $T$. An element $y \in Y$ is called a \emph{generic point} of $Y$ if $Y = \overline{\{y\}}$. Proposition \ref{closure} shows that every prime submodule element $p$ of $M$ is a generic point of the irreducible closed subset $V^{*}(p)$ in $X^{M}$. Now we prove that every irreducible closed subset of $X^{M}$ has a generic point.
\begin{theorem}     \label{bijection}  \label{generic point}
Let $_{R}M$ be an le-module and the natural map $\psi: X^{M}\rightarrow X^{\overline{R}}$ be surjective. Then the following statements hold.
\begin{enumerate}
\item[(i)]
Then $Y \subseteq X^M$ is an irreducible closed subset of $X^{M}$ if and only if $Y = V^{*}(p)$ for some $p \in X^{M}$. Every irreducible closed subset of $X^{M}$ has a generic point.
\item[(ii)]
The correspondence $p \mapsto V^{*}(p)$ is a surjection of $X^{M}$ onto the set of irreducible closed subsets of $X^{M}$.
\item[(iii)]
The correspondence $V^{*}(p) \mapsto \overline{(p:e)}$ is a bijection of the set of irreducible components of $X^{M}$ onto the set of minimal prime ideals of $\overline{R}$.
\end{enumerate}
\end{theorem}
\begin{proof}
(i) Let $Y$ be an irreducible closed subset of $X^{M}$. Then $Y = V^{*}(n)$ for some submodule element $n$ of $M$. By Proposition \ref{irreducible-prime}, we have $(\Im(V^{*}(n)):e) = (\Im(Y):e) = P$ is a prime ideal of $R$. Then $P/Ann(M) \in X^{\overline{R}}$. Since $\psi$ is surjective, there exists a prime submodule element $p$ of $X^{M}$ such that $\psi(p) = P/Ann(M)$, i.e, $(p:e)/Ann(M) = P/Ann(M)$ and hence $(p:e) = P = (\Im(V^{*}(n)):e)$. Therefore $V^{*}(\Im(V^{*}(n))) = V^{*}(p)$. Since $V^{*}(n)$ is closed, $V^{*}(\Im(V^{*}(n))) = V^{*}(p)$ implies $V^{*}(n) = V^{*}(p)$, by Proposition \ref{closed}. Thus $Y = V^{*}(n) = V^{*}(p)$.

Converse part follows from Corollary \ref{irreducible closed subset}.

Now it follows from $V^{*}(p) = \overline{\{p\}}$ that every irreducible closed subset of $X^{M}$ has a generic point.\\
(ii) Follows from (i). \\
(iii) First assume that $V^{*}(p)$ is an irreducible component of $X^{M}$. Then $V^{*}(p)$ is maximal in the set $\{V^{*}(q): q \in X^{M} \}$. Clearly $\overline{(p:e)}$ is a prime ideal of $\overline{R}$. To show $\overline{(p:e)}$ is minimal let $\overline{P}$ be a prime ideal of $\overline{R}$ such that $\overline{P} \subseteq \overline{(p:e)}$. Then $P \subseteq (p:e)$. Since $\psi$ is surjective there exists $q \in X^{M}$ such that $\psi(q) = \overline{P}$, i.e, $\overline{(q:e)} = \overline{P}$. Therefore $(q:e) = P \subseteq (p:e)$ and hence $V^{*}(p) \subseteq V^{*}(q)$. Since $V^{*}(p)$ is maximal in the set $\{V^{*}(q): q \in X^{M} \}$, $V^{*}(p) = V^{*}(q)$. Thus $(p:e) = (q:e) = P$ and so $\overline{P} = \overline{(p:e)}$.

Next let $\overline{P}$ be a minimal prime ideal of $\overline{R}$. Since $\psi$ is surjective there exists $p \in X^{M}$ such that $(p:e) = P$. From Corollary \ref{irreducible closed subset}, $V^{*}(p)$ is an irreducible subset of $X^{M}$. Now let $q \in X^{M}$ be such that $V^{*}(p) \subseteq V^{*}(q)$. Then $(q:e) \subseteq (p:e) = P$ which implies that $\overline{(q:e)} \subseteq \overline{P}$. Since $\overline{P}$ is minimal, $\overline{(q:e)} = \overline{P}$. Thus $(q:e) = P = (p:e)$ and so $V^{*}(p) = V^{*}(q)$. Therefore $V^{*}(p)$ is an irreducible component of $X^{M}$. This completes the proof.
\end{proof}

\section{$X^{M}$ as a spectral space}   \label{section spectral}
A topological space $T$ is called \emph{spectral} if it is $T_{0}$, quasi-compact, the quasi-compact open subsets of $T$ are closed under finite intersection and form an open basis, and each irreducible closed subset of $T$ has a generic point. Importance of spectral topological space is that a topological space $T$ is homeomorphic to Spec$(R)$ for some commutative ring $R$ if and only if $T$ is spectral. Any closed subset of a spectral topological space with the induced topology is spectral.

From Theorem \ref{open base} and Theorem \ref{generic point}, it follows that $X^M$ satisfies all the conditions to be a spectral space except being $T_0$.

Now we prove some equivalent characterizations for $X^M$ to be spectral.
\begin{theorem}   \label{T_{0} space} \label{spectral}
Let $_{R}M$ be an le-module and the natural map $\psi: X^{M} \rightarrow X^{\overline{R}}$ be surjective. Then the following conditions are equivalent:
\begin{enumerate}
\item[(i)] $X^{M}$ is a spectral space;
\item[(ii)] $X^{M}$ is a $T_{0}$-space;
\item[(iii)] For every $p,q \in X^{M}$, $V^{*}(p) = V^{*}(q)$ implies that $p =q$;
\item[(iv)] $|Spec_P(M)| \leqslant 1$ for every $p \in Spec(R)$;
\item[(v)] $\psi$ is injective;
\item[(vi)] $X^{M}$ is homeomorphic to $X^{\overline{R}}$.
\end{enumerate}
\end{theorem}
\begin{proof}
Equivalence of (iii), (iv) and (v) follows from Proposition \ref{injective} and equivalence of (v) and (vi) follows from Corollary \ref{homeomorphic}.
\\ $(i) \Rightarrow (iii)$: Let $p, q \in X^M$ be such that $V^*(p)=V^*(q)$. Then, by Proposition \ref{closure}, $\overline{\{p\}}=V^*(p)=V^*(q) = \overline{\{q\}}$. Since $X^M$ is spectral, it is $T_0$. Hence the closures of distinct points in $X^M$ are distinct and it follows that $p=q$.
\\ $(iii) \Rightarrow (ii)$: Let $p, q \in X^M$ be two distinct prime submodule elements. Then $V^*(p) \neq V^*(q)$ which implies that $\overline{\{p\}} \neq \overline{\{q\}}$, by Proposition \ref{closure}. Hence $X^M$ is $T_0$.
\\ $(ii) \Rightarrow (i)$: This follows from Theorem \ref{open base} and Theorem \ref{generic point}.
\end{proof}

An le-module $_{R}M$ is called a \emph{multiplication le-module} if every submodule element $n$ of $M$ can be expressed as $n = Ie$, for some ideal $I$ of $R$. Let $n$ be a submodule element of a multiplication le-module $_{R}M$. Then there exists an ideal $I$ of $R$ such that $n = Ie$. For each $k \in \mathbb{N}$ and $a_{i} \in (n:e)$, we have
\begin{center}
$a_{1}e+a_{2}e+\cdots+a_{k}e \leqslant n+n+\cdots+n (k times) = n$
\end{center}
Hence $(n:e)e \leqslant n$. Also $n = Ie$ implies $I \subseteq (n:e)$ which implies that $n = Ie \leqslant (n:e)e$. Thus $(n:e)e = n$.

\begin{theorem}
Let $_{R}M$ be a multiplication le-module and the natural map $\psi: X^{M} \rightarrow X^{\overline{R}}$ be surjective. Then $X^{M}$ is a spectral space.
\end{theorem}
\begin{proof}
First we show that $X^{M}$ is $T_{0}$. Let $l$ and $n$ be two distinct elements of $X^{M}$. If possible let $V^{*}(l) = V^{*}(n)$. Then $(l:e) = (n:e)$ which implies that $l = (l:e)e = (n:e)e = n$, a contradiction. Thus $\overline{\{l\}} = V^{*}(l) \neq V^{*}(n) = \overline{\{n\}}$, by Proposition \ref{closure}. Hence $X^{M}$ is $T_{0}$. Thus the theorem follows from Theorem \ref{open base} and Theorem \ref{generic point}.
\end{proof}

Thus assuming surjectivity of the natural map $\psi: X^{M} \longrightarrow X^{\overline{R}}$ gives us several equivalent characterizations for $X^M$ to be spectral. In other words, surjectivity of $\psi$ is a very strong condition to imply that $X^M$ is spectral. Now we prove some conditions for $X^M$ to be spectral, which do not assume surjectivity of the natural map $\psi: X^{M} \longrightarrow X^{\overline{R}}$.
\begin{theorem}
Let $_{R}M$ be an le-module and $\psi: X^{M} \longrightarrow X^{\overline{R}}$ be the natural map of $X^{M}$ such that the image Im$\; \psi$ of $\psi$ is a closed subset of $X^{\overline{R}}$. Then $X^{M}$ is a spectral space if and only if $\psi$ is injective.
\end{theorem}
\begin{proof}
Let $Y = Im \; \psi$, Since $X^{\overline{R}}$ is a spectral space and $Y$ is a closed subset of $X^{\overline{R}}$, $Y$ is spectral for the induced topology. Suppose that $\psi$ is injective. Then $\psi : X^{M} \rightarrow Y$ is a bijection and by Proposition \ref{continuous}, $\psi : X^{M} \rightarrow Y$ is continuous. Now we show that $\psi$ is closed. Let $V^{*}(n)$ be a closed subset of $X^{M}$ for some submodule element $n$ of $M$. Then $Y' = Y \cap V^{\overline{R}}(\overline{(n:e)})$ is a closed subset of $Y$. Also $\psi^{-1}(Y') = \psi^{-1}(Y \cap V^{\overline{R}}(\overline{(n:e)})) = \psi^{-1}(Y) \cap \psi^{-1}(V^{\overline{R}}(\overline{(n:e)})) = X^{M} \cap \psi^{-1}(V^{\overline{R}}(\overline{(n:e)})) = \psi^{-1}(V^{\overline{R}}(\overline{(n:e)})) = V((n:e)e) = V^{*}(n)$ by Proposition \ref{continuous}. Since $\psi : X^{M} \rightarrow Y$ is surjective, $\psi(V^{*}(n)) = \psi(\psi^{-1}(Y')) = Y'$, a closed subset of $Y$. Thus $\psi : X^{M} \rightarrow Y$ is a homeomorphism and hence $X^{M}$ is a spectral space. Conversely, let $X^{M}$ be a spectral space. Then $X^{M}$ is a $T_{0}$-space and so $\psi$ is injective by Theorem \ref{T_{0} space}.
\end{proof}

\begin{theorem}
Let $_{R}M$ be an le-module such that $Spec(M) = X^{M}$ is a non-empty finite set. Then $X^{M}$ is a spectral space if and only if $|Spec_{P}(M)| \leqslant 1$ for every $P \in Spec(R)$.
\end{theorem}
\begin{proof}
The finiteness of $|X^{M}|$ implies that $X^{M}$ is quasi-compact and the quasi-compact open subsets of $X^{M}$ are closed under finite intersection and forms an open base. We show that every irreducible closed subset of $X^{M}$ has a generic point. Let $Y = \{p_{1}, p_{2}, \cdots, p_{n} \}$ be an irreducible closed subset of $X^{M}$. Then $Y = \overline{Y} = V^{*}(\Im(Y)) = V^{*}(p_{1} \wedge p_{2} \wedge \cdots \wedge p_{n}) = V^{*}(p_{1}) \cup V^{*}(p_{2}) \cup \cdots \cup V^{*}(p_{n}) = \overline{\{p_{1}\}} \cup \overline{\{p_{2}\}} \cup \cdots \cup \overline{\{p_{n}\}}$. Since $Y$ is irreducible, $Y = \overline{\{p_{i}\}}$ for some $i$. Therefore by Hochster`s characterization of spectral spaces, $X^{M}$ is a spectral space if and only if $X^{M}$ is a $T_{0}$-space. Hence by Theorem \ref{T_{0} space}, it follows that $X^{M}$ is a spectral space if and only if $|Spec_{P}(M)| \leqslant 1$ for every $P \in Spec(R)$.
\end{proof}

\bibliographystyle{amsplain}

\end{document}